\documentclass[11pt]{article}
\usepackage{amsmath,amssymb,amsthm,mathtools,geometry,booktabs}
\usepackage{microtype}
\usepackage{hyperref}
\title{\textbf{A Sieve on Rational Imbalances and the First Appearance of Denominators}}
\author{Paul Alexander Bilokon\footnote{Department of Mathematics, Imperial College London, South Kensington Campus, London SW7 2AZ. Email:~\textsf{paul.bilokon@imperial.ac.uk}}}
\date{26 October 2025}

\theoremstyle{plain} 
\newtheorem{theorem}{Theorem}[section]
\newtheorem{lemma}[theorem]{Lemma}
\newtheorem{corollary}[theorem]{Corollary}

\theoremstyle{definition} 
\newtheorem{definition}[theorem]{Definition}

\theoremstyle{remark} 

\begin{document}
\maketitle

\begin{abstract}
We construct a sieve that enumerates rational ``imbalances'' of the form $(p-q)/(p+q)$ for integers $p\ge2$ and $1\le q<p$, ordered lexicographically by $(p,q)$. Each imbalance is reduced to lowest terms, and we record the sequence of distinct denominators as they first appear.

We show that every positive integer occurs exactly once as such a denominator, and that its first appearance coincides with the unit fraction $1/d$.

We then prove that the sieve, when viewed as a map from pairs $(p,q)$ to reduced fractions, enumerates all rational numbers in $(-1,1)$ without repetition, extend it symmetrically to all of $\mathbb{Q}$, and discuss its connections to hyperbolic geometry and rational enumeration theory.
\end{abstract}

\section{Introduction}

The rational numbers have long inspired attempts to capture their arithmetic and geometric structure through explicit orderings and enumerations.  
From the \emph{Farey sequences} of Cauchy and Stern in the early nineteenth century~\cite{Farey1816,Cauchy1840,Stern1858}, to the mediant-based constructions of Brocot \cite{Brocot1861} and the elegant binary-tree formulation of Calkin and Wilf \cite{CalkinWilf2000}, each approach reveals a different facet of the intrinsic symmetry of $\mathbb{Q}$.  
These enumerations have subsequently been interpreted in the context of modular transformations, continued fractions, and dynamical systems on the projective line; see, for example, works by Klein~\cite{Klein1928}, Dedekind~\cite{Dedekind1877}, and more recently Conway and Guy~\cite{ConwayGuy1996}.

The present note proposes a closely related sieve, derived not from the mediant operation but from the \emph{bilinear form}
\[
(p,q)\ \longmapsto\ \frac{p-q}{p+q},
\]
which may be viewed as a rational incarnation of the Cayley transform linking the real line to the open interval $(-1,1)$.  
This transform, already implicit in the parametrisation of the unit hyperbola and the theory of M{\"o}bius transformations, yields a natural and elementary mapping of the integer lattice into the rationals.

Despite its simplicity, the sieve exhibits an unexpected regularity.  
When the pairs $(p,q)$ are processed lexicographically, every denominator arises exactly once, and its first appearance corresponds to the unit fraction $1/d$.  
Moreover, the indices of these first appearances follow two interleaved quadratic progressions---one for odd and one for even denominators---together forming a permutation of the positive integers.  

The purpose of this paper is to describe this sieve precisely, to prove its completeness and uniqueness properties, and to discuss its relation to classical enumerations, hyperbolic geometry, and the modular symmetries of $\mathrm{PSL}(2,\mathbb{Z})$.  
While entirely elementary, the result may serve as a reminder that even within the well-trodden domain of rational enumeration, new algebraic structure can still emerge from the most basic arithmetic identities.

\section{The Imbalance Sieve}

Let $p$ and $q$ be integers satisfying $p\ge 2$ and $1\le q<p$.  
Define the \emph{imbalance}
\begin{equation}
\delta(p,q)=\frac{p-q}{p+q}.
\end{equation}

We process these pairs in \emph{lexicographic order}:
\[
(2,1),\ (3,1),\ (3,2),\ (4,1),\ (4,2),\ (4,3),\ (5,1),\dots
\]
That is, $(p_1,q_1)$ precedes $(p_2,q_2)$ if either $p_1<p_2$ or $p_1=p_2$ and $q_1<q_2$.

\begin{definition}[Iteration index]
Each pair $(p,q)$ is assigned the \emph{iteration index}
\begin{equation}
i(p,q)=\frac{(p-2)(p-1)}{2}+(q-1),
\end{equation}
so that $i(2,1)=0$, $i(3,1)=1$, $i(3,2)=2$, etc.
\end{definition}

\begin{definition}[Reduced imbalance]
Each $\delta(p,q)$ is expressed in lowest terms as
\begin{equation}
\delta(p,q)=\frac{a(p,q)}{d(p,q)},\qquad \gcd(a(p,q),d(p,q))=1.
\end{equation}
We refer to $d(p,q)$ as the \emph{reduced denominator}.
\end{definition}

The sieve proceeds by enumerating these reduced fractions in increasing order of $i(p,q)$
and recording new denominators as they first appear.

\section{Basic Arithmetic Structure}

\begin{lemma}\label{lem:gcd}
For integers $p,q$ with $p>q>0$,
\[
\gcd(p+q,p-q)=\gcd(p+q,2q)\in\{1,2\}.
\]
\end{lemma}

\begin{proof}
Because $(p+q)-(p-q)=2q$, any common divisor of $p+q$ and $p-q$ divides $2q$.
Conversely, a divisor of $p+q$ and $2q$ divides their difference $(p+q)-2q=p-q$.
The parity constraint implies the gcd is either $1$ or $2$.
\end{proof}

\begin{corollary}\label{cor:cases}
The reduced denominator $d(p,q)$ of $\delta(p,q)$ is
\[
d(p,q)=
\begin{cases}
p+q, & \text{if }p+q\text{ is odd},\\[3pt]
\dfrac{p+q}{2}, & \text{if }p+q\text{ is even}.
\end{cases}
\]
\end{corollary}

Hence denominators arise in two families:
\begin{itemize}
\item the \textbf{odd case} $d=p+q$;
\item the \textbf{even case} $d=(p+q)/2$.
\end{itemize}

\section{Earliest Occurrence of a Denominator}

\begin{definition}
The \emph{first appearance index} of denominator $d$ is
\[
I(d)=\min\{\,i(p,q):\,d(p,q)=d\,\}.
\]
\end{definition}

Since $i(p,q)$ increases with both $p$ and $q$, the earliest appearance corresponds
to the smallest possible reduced numerator, namely $1$.
We treat the odd and even cases separately.

\subsection{Odd denominators}

For odd $d$, $d=p+q$ and $\gcd(p+q,p-q)=1$.
Setting $p-q=1$ gives
\[
p=\frac{d+1}{2},\qquad q=\frac{d-1}{2}.
\]
Then $\delta(p,q)=1/d$ and
\begin{align}
I_{\text{odd}}(d)
&= i\!\left(\frac{d+1}{2},\frac{d-1}{2}\right)
= \frac{(p-2)(p-1)}{2}+q-1\\
&=\frac{(p-2)(p+1)}{2},\qquad p=\tfrac{d+1}{2}.
\end{align}
That is,
\begin{equation}
I_{\text{odd}}(d)=\frac{d^2-9}{8}.
\end{equation}

\subsection{Even denominators}

For even $d$, reduction gives $(p+q)/2=d$ and $\gcd(p+q,p-q)=2$.
Choosing the minimal positive $p-q=2$ yields
\[
p=d+1,\qquad q=d-1,
\]
hence $\delta(p,q)=1/d$ and
\begin{align}
I_{\text{even}}(d)
&=i(d+1,d-1)
=\frac{(d-1)d}{2}+(d-1)-1
=\frac{d(d+1)}{2}-2.
\end{align}

\section{First Appearance Theorem}

\begin{theorem}[First Appearance of Denominators]\label{thm:main}
For each integer $d\ge2$:
\begin{enumerate}
\item The first imbalance whose reduced denominator equals $d$ is the unit fraction $\delta(p,q)=1/d$.
\item The iteration index of this occurrence is
\[
I(d)=
\begin{cases}
\dfrac{d^2-9}{8}, & d\text{ odd},\\[8pt]
\dfrac{d(d+1)}{2}-2, & d\text{ even}.
\end{cases}
\]
\item As $d$ increases through $\mathbb{Z}_{>0}$, the denominators appear in the order
\[
3,2,5,7,4,9,11,6,13,15,8,17,19,10,21,23,12,\dots,
\]
and every positive integer occurs exactly once.
\end{enumerate}
\end{theorem}

\section{Completeness of the Enumeration}

The mapping
\[
\Phi:\{(p,q):p>q>0\}\to(-1,1)\cap\mathbb{Q},\qquad
\Phi(p,q)=\frac{p-q}{p+q},
\]
is bijective up to the sign symmetry $\Phi(p,-q)=-\Phi(p,q)$.

\subsection{Surjectivity}
Given $r=a/d\in(-1,1)\cap\mathbb{Q}$ in lowest terms, define $p=\frac{d+a}{2}$, $q=\frac{d-a}{2}$ if $a,d$ share parity, or $(p,q)=(d+a,d-a)$ otherwise.
Then $\Phi(p,q)=a/d$. Hence $\Phi$ is surjective.

\subsection{Injectivity}
If $\Phi(p_1,q_1)=\Phi(p_2,q_2)$, the defining equation forces $(p_1,q_1)$ and $(p_2,q_2)$ to be proportional. Parity constraints eliminate nontrivial scaling, yielding $(p_1,q_1)=(p_2,q_2)$. Thus $\Phi$ is injective.

\section{Symmetry and Extension to All Rationals}

The image of $\Phi$ is symmetric under
\[
N:x\mapsto -x,\qquad I:x\mapsto \frac{1}{x},
\]
corresponding to $(p,q)\mapsto(q,p)$ and $(p,q)\mapsto(p,-q)$, respectively.

\begin{lemma}[Cayley transform]
The map
\[
\mathcal{C}:(-1,1)\to(0,\infty),\qquad \mathcal{C}(x)=\frac{1+x}{1-x},
\]
is a bijection with inverse $\mathcal{C}^{-1}(r)=\frac{r-1}{r+1}$.
\end{lemma}

\begin{theorem}[Extension to $\mathbb{Q}$]
Let $(x_n)_{n\ge0}$ enumerate $(-1,1)\cap\mathbb{Q}$ via the sieve.
Define
\[
r_n^{+}=\frac{1+x_n}{1-x_n},\qquad r_n^{-}=-\,\frac{1+x_n}{1-x_n}.
\]
Then
\[
0,\ \pm1,\ r_0^{+},r_0^{-},r_1^{+},r_1^{-},r_2^{+},r_2^{-},\dots
\]
enumerates all of $\mathbb{Q}$ without repetition.
\end{theorem}

\section{Discussion and Perspectives}

The imbalance sieve provides a new and remarkably elementary path to enumerating the rationals.

\subsection*{1. Geometric interpretation}
The mapping $(p,q)\mapsto(p-q)/(p+q)$ is the algebraic analogue of the hyperbolic tangent:
\[
\tanh\theta = \frac{e^{\theta}-e^{-\theta}}{e^{\theta}+e^{-\theta}}.
\]
Hence the sieve discretizes the hyperbolic line by rational ``rapidities''.
Each pair $(p,q)$ represents a rational point on the unit hyperbola $x^2-y^2=1$, and the sieve orders these points lexicographically by their homogeneous coordinates.

\subsection*{2. Relation to classical enumerations}
Via the Cayley transform
\[
\frac{p-q}{p+q} \mapsto \frac{p}{q},
\]
the sieve linearizes the structure underlying the Calkin--Wilf and Stern--Brocot trees.
It therefore provides a direct algebraic parametrization of the same rational set,
without recourse to continued fractions or mediant recursions.

\subsection*{3. Analytical structure}
The explicit quadratic index laws
\[
I_{\text{odd}}(d)=\frac{d^2-9}{8},\qquad
I_{\text{even}}(d)=\frac{d(d+1)}{2}-2
\]
yield two interleaved arithmetic progressions forming a permutation of $\mathbb{Z}_{>0}$.
This structure suggests asymptotic density $\rho(i)\sim i^{-1/2}$ of new denominators,
a feature potentially useful for probabilistic studies of rational distributions.

\subsection*{4. Dynamical and group-theoretic aspects}
The generating symmetries
\[
x\mapsto -x,\qquad x\mapsto 1/x,\qquad x\mapsto \frac{1+x}{1-x}
\]
generate the modular group $PSL(2,\mathbb{Z})$.
The sieve thus gives an explicit orbit decomposition of $\mathbb{Q}$ under this group,
linking a purely enumerative process to the dynamics of modular transformations.

\subsection*{5. Future directions}
Further research could explore:
\begin{itemize}
\item The measure-theoretic limit of $\delta(p,q)$ as $(p,q)$ range over large rectangles;
\item The distribution of first-appearance indices and their connection to divisor sums;
\item Extensions to Gaussian rationals or to $p$-adic analogues;
\item Algorithmic applications in rational sampling and symbolic coding.
\end{itemize}

\paragraph{Summary.}
The imbalance sieve, built solely on parity and lexicographic ordering, reveals
an unexpectedly regular algebraic structure hidden within the rationals.
It unites combinatorial enumeration, hyperbolic geometry, and modular symmetry
in a single transparent construction.

\section*{Acknowledgments}

We would like to thank Michel Marcus for his corrections.


\begin{thebibliography}{Cau40}

\bibitem[Bro61]{Brocot1861}
Achille Brocot.
\newblock Calcul des rouages par approximation, nouvelle m{\'e}thode.
\newblock {\em Revue Chronom{\'e}trique}, 3:186--194, 1861.
\newblock First publication describing mediant-based rational approximation.

\bibitem[Cau40]{Cauchy1840}
Augustin-Louis Cauchy.
\newblock Sur les moyens d'obtenir, par des proc{\'e}d{\'e}s d'analyse, les valeurs approch{\'e}es des racines des {\'e}quations num{\'e}riques, et les valeurs approch{\'e}es des nombres irrationnels.
\newblock {\em Comptes Rendus Hebdomadaires des S{\'e}ances de l'Acad{\'e}mie des Sciences}, 10:560--568, 1840.
\newblock Contains early remarks on mediant fractions and rational approximation.

\bibitem[CG96]{ConwayGuy1996}
John~H. Conway and Richard~K. Guy.
\newblock {\em The Book of Numbers}.
\newblock Springer-Verlag, New York, 1996.

\bibitem[CW00]{CalkinWilf2000}
Neil Calkin and Herbert~S. Wilf.
\newblock Recounting the rationals.
\newblock {\em The American Mathematical Monthly}, 107(4):360--363, 2000.

\bibitem[Ded77]{Dedekind1877}
Richard Dedekind.
\newblock {\em Gesammelte mathematische Werke. Vol.~1}.
\newblock Friedrich Vieweg und Sohn, Braunschweig, 1877.
\newblock Contains ``Erkl{\"a}rung der modularen Funktionen'' and other papers on the modular group.

\bibitem[Far16]{Farey1816}
John Farey.
\newblock On a curious property of vulgar fractions.
\newblock {\em Philosophical Magazine}, 47:385--386, 1816.
\newblock Letter to the editor, introducing what became known as the Farey sequence.

\bibitem[Kle28]{Klein1928}
Felix Klein.
\newblock {\em Vorlesungen {\"u}ber die Entwicklung der Mathematik im 19. Jahrhundert}.
\newblock Springer, Berlin, 1928.

\bibitem[Ste58]{Stern1858}
Moritz~A. Stern.
\newblock Ueber eine zahlentheoretische Funktion.
\newblock {\em Journal f{\"u}r die reine und angewandte Mathematik (Crelle)}, 55:193--220, 1858.

\end{thebibliography}

\section*{Appendix A. Sample Table of the Imbalance Sieve}

\begin{center}
\begin{tabular}{@{}rrrrrcl@{}}
\toprule
Index $i$ & $p$ & $q$ & $\delta(p,q)$ & Reduced Form & New Denominator? & First Appearance \\ \midrule
0 & 2 & 1 & $1/3$ & $1/3$ & Yes & $d=3$\\
1 & 3 & 1 & $1/2$ & $1/2$ & Yes & $d=2$\\
2 & 3 & 2 & $1/5$ & $1/5$ & Yes & $d=5$\\
3 & 4 & 1 & $3/5$ & $3/5$ & No &  \\
4 & 4 & 2 & $1/3$ & $1/3$ & No &  \\
5 & 4 & 3 & $1/7$ & $1/7$ & Yes & $d=7$\\
6 & 5 & 1 & $2/3$ & $2/3$ & No & \\
7 & 5 & 2 & $3/7$ & $3/7$ & No & \\
8 & 5 & 3 & $1/4$ & $1/4$ & Yes & $d=4$\\
9 & 5 & 4 & $1/9$ & $1/9$ & Yes & $d=9$\\
10 & 6 & 1 & $5/7$ & $5/7$ & No & \\
11 & 6 & 2 & $1/2$ & $1/2$ & No & \\
12 & 6 & 3 & $1/3$ & $1/3$ & No & \\
13 & 6 & 4 & $1/5$ & $1/5$ & No & \\
14 & 6 & 5 & $1/11$ & $1/11$ & Yes & $d=11$\\
\bottomrule
\end{tabular}
\end{center}

The ``new denominator'' column records when a denominator $d$ appears for the first time.
The sequence of these first appearances begins
\[
3,2,5,7,4,9,11,6,13,15,8,17,19,10,21,23,12,\dots
\]
and continues without repetition.

\end{document}